\documentclass[12pt,a4paper]{amsart}

\usepackage{amsfonts, amsmath, amssymb, amsthm, amscd, hyperref}
\usepackage{a4wide}
\newtheorem*{thm}{Theorem}
\newtheorem*{corollary}{Corollary}
\newtheorem*{remark}{Remark}

\def\char{\rm{char}\,}

\author{Fedor Petrov}
\thanks{The work is supported by RFBR grants 11-01-00677-a, 13-01-00935-a, 13-01-12422-ofi-m and
President of Russia grant MK-6133.2013.1.}
\email{fedyapetrov@gmail.com}
\address{
St. Petersburg   Department   of   V.~A.~Steklov    Institute   of   Mathematics
RAS, St. Petersburg State University, Yaroslavl State Univeristy.}

\title{Combinatorial Nullstellensatz Approach to Polynomial Expanding}

\begin{document}

\maketitle

Let $\mathbb{F}$  be a field and $f(x,y)\in \mathbb{F}[x,y]$ be a polynomial of two variables.
For non-empty sets $A,B\subset \mathbb{F}$ denote
$$
f(A,B)=\{f(x,y):x\in A, y\in B\}.
$$
There are numerous works concerning estimates of $|f(A,B)|$
in terms of $|A|$ and $|B|$ for various polynomials $f$. Probably, the first
result in this area is Cauchy-Davenport theorem, stating that for $f(x,y)=x+y$
and $\mathbb{F}=\mathbb{F}_p$ for prime $p$ one has $|f(A,B)|\geq \min(|A|+|B|-1,p)$.
Combinatorial Nullstellensatz of Alon \cite{CN} is one of most flexible ways
to prove Cauchy-Davenport theorem. In particular, it easily generalizes to
restricted sumsets estimates like Erdos-Heilbronn conjecture (unlike purely
combinatorial methods).

There are many asymptotical results for other polynomial $f$. Say, Bourgain \cite{B}
proved that for $f(x,y)=x^2+xy$,
given $\alpha\in (0,1)$
there exists $\beta>\alpha$ such that ford $\mathbb{F}=\mathbb{F}_p$ (here $p$ is large enough prime), and
$|A|,|B|\geq p^{\alpha}$ one has $|f(A,B)|\geq p^{\beta}$. This phenomenon (the estimate is asymptotically much better
then in Cauchy-Davenport case) is called polynomial expanding. It is intimately connected to sum-products
estimates and was intensively studied in recent papers, some of which are mentioned in the bibliography.
The main methods are spectral graph theory
and Fourier analysis. Tao in recent paper \cite{TAO} uses also some algebraic geometry.

The aim of this paper is to give a proof of some weak (of Cauchy-Davenport type) estimate for the Bourgain-type
expanders $g(x)+yh(x)$. The possible advantage of this result is that estimates are very explicit
(without implicit asymptotical constants) and say something for all fields.

Our proof is in spirit of Combinatorial Nullstellensatz. However, we do not use it as a blackbox, but
involve the proof.

\begin{thm}
Let $\mathbb{F}$ be a field, $g(x),h(x)$ be polynomials in $\mathbb{F}[x]$,
$A$ and $B$ be non-empty finite subsets of a $\mathbb{F}$, $|A|=a$, $|B|=b$.
Assume also that $d=\deg g(x)>\deg h(x)$ and $A$ does not contain roots of $h(x)$.
Assume further that $k\leq (a-1)/d+b-1$ and binomial coefficient
$\binom{k}{b-1}$ does not vanish in $\mathbb{F}$.
Then
$$
\left|\{g(x)+yh(x):x\in A,y\in B\}\right|>k.
$$
\end{thm}

The theorem immediately yields the following
\begin{corollary}
Let $p=\char \mathbb{F}$ (and $p=\infty$ if $\char \mathbb{F}=0$).
Then
$$
\left|\{g(x)+yh(x):x\in A,y\in B\}\right|\geq \min(a/d+b-1,p).
$$
In particular, for Bourgain's expander we get $|\{x^2+xy:x\in A,y\in B\}|\geq \min(a/2+b-1,p)$
provided that $0\notin A$. 
\end{corollary}

\begin{proof}[Proof of the theorem]
Assume the contrary. Then there exists a set $C$ of cardinality $k$ so that
$g(x)+yh(x)\in C$ for all $x\in A$, $y\in B$. Clearly $k\geq b$ (just fix $x$ and vary $y$).
Denote
$$
P(x,y):=\prod_{c\in C} (g(x)+yh(x)-c)=\sum_{i,j} \lambda_{i,j} g(x)^ih(x)^jy^j
$$
for some pairs $(i,j)$ of non-negative integers and some coefficients $\lambda_{i,j}$
in $\mathbb{F}$. Such a polynomial $P(x,y)$ vanishes on $A\times B$.
Consider some $\mathbb{F}$-valued functions $\alpha(x),\beta(y)$ defined on $A$ and $B$
respectively. Look at the following sum, which eventually vanishes
\begin{multline}
\label{main}
\sum_{x\in A,y\in B} \alpha(x)\beta(y) P(x,y)=\sum_{i,j}\lambda_{i,j}
\sum_{x\in A,y\in B} \alpha(x)\beta(y) g(x)^ih(x)^jy^j=\\
\sum_{i,j}\lambda_{i,j} \left(\sum_{x\in A} \alpha(x) g(x)^ih(x)^j \right)\left(\sum_{y\in B} \beta(y) y^j
\right).
\end{multline}
Our goal is to choose functions $\alpha,\beta$ so that there exists unique non-zero term in the
last expression in (\ref{main}).
Let's choose $\beta$ so that
$$
\sum_{y\in B} \beta(y) y^j =
\begin{cases} 0 &\mbox{if } 0\leq j\leq b-2\\
1& \mbox{if } j=b-1. \end{cases}
$$
Such $\beta$ does exist, since the Vandermonde determinant for the set $B$
does not vanish.
Then all terms in (\ref{main}) with $j<b-1$ do vanish.
If $j\geq b-1$, then we may expand
$$
g(x)^ih(x)^j=h(x)^{b-1} \sum_{\nu=0}^{d(k-b+1)} \eta_{i,j}(\nu) x^{\nu}.
$$
Let's choose
$\alpha$ so that
$$
\sum_{x\in A} \alpha(x)h(x)^{b-1} x^i =
\begin{cases} 0 &\mbox{if } 0\leq i< d(k-b+1)\\
1& \mbox{if } i=d(k-b+1). \end{cases}
$$
Since $d(k-b+1)\leq a-1$, this is a (part of) Vandermonde system again
(for unknowns $\alpha(x)\cdot h(x)^{b-1}$), and therefore has a solution.
For this choice of $\alpha$ all summands
$$
\sum_{x\in A} \alpha(x)h(x)^{b-1} \eta_{i,j}(\nu) x^{\nu}
$$
corresponding to fixed $i,j$ and fixed $\nu<d(k-b+1)$ vanish.
Now note that $\eta_(i,j)(d(k-b+1))=0$ unless $j=b-1$, $i=k-b+1$
(here we use that $\deg h(x)<d$). And if $j=b-1$, $i=k-b+1$, we have
$$
\eta_{k-b+1,b-1}\left(d(k-b+1)\right)=\binom{k}{b-1}M^{k-b+1},
$$
where $M$ is leading coefficient of the polynomial $g(x)$.
So, by our assumption this expression does not vansih in $\mathbb{F}$. After all, we indeed
have unique non vanishing term in (\ref{main}), as desired.
\end{proof}

\begin{remark} Let $\mathbb{F}$ be a field of $p^n$ elements for prime $p$, $B$ be any subfield
of, say, $p^m$ elements, and $A=B\setminus \{0\}$. Then $f(A,B)=B$ for any polynomial $f$ and we get no non-trivial
bound. But already for $|B|=b=p^m+1$ and $|A|=a\geq C\cdot p^m$, $0<C<1$, for, say, $f(x,y)=x^2+xy$, we get an estimate
$|f(A,B)|\geq (1+C/2)p^m-1$, since corresponding binomial coefficient is not divisible by $p$.
It would be interesting to have structured version of this result, i.e. to prove that if $|f(A,B)|$
is close to $|B|$, then $B$ is close to a subfield. Also, the constant $1+C/2$ does not look to be sharp
and probably the correct constant is $1+C$.
\end{remark}

I am grateful to Ilya Shkredov for pending my attention on this question and many fruitful discussions.


\begin{thebibliography}{99}
\bibitem{CN} N.~Alon. Combinatorial Nullstellensatz. Combin. Probab. Comput. 8 (1999), 7--29.

\bibitem{B}  J.~Bourgain.
More on the sum-product phenomenon in prime fields and its applications. Int. J. Number Theory 1 (2005), no. 1, 1--32.

\bibitem{HLS} D.~Hart, L.~Li and C.-Y.~Shen. Fourier analysis and expanding phenomena in finite fields.
Proc. Amer. Math. Soc. 141 (2013), 461--473

\bibitem{TAO} T.~Tao. Expanding polynomials over finite fields of large characteristic, and a regularity lemma for definable sets.
	arXiv:1211.2894

\bibitem{Vu} V.~H.~Vu. Sum-product estimates via directed expanders. Math. Res. Lett. 15(2) (2008), 375--388.

\bibitem{Shkredov} I.~D.~Shkredov. On monochromatic solutions of some nonlinear equations in $\mathbb{Z}/p\mathbb{Z}$.
Math. Notes, 88 (2010), no. 3-4, 603--611.

\end{thebibliography}
\end{document}